\documentclass[11pt, onesided, reqno]{amsart}

\usepackage{amssymb}
\usepackage{amsmath}
\usepackage{color}
\usepackage{enumerate}
\usepackage{graphicx}

\evensidemargin\oddsidemargin

\newtheorem{theorem}{Theorem}
\newtheorem{proposition}[theorem]{Proposition}
\newtheorem{corollary}[theorem]{Corollary}
\newtheorem{lemma}[theorem]{Lemma}

\theoremstyle{definition}

\newtheorem{remark}[theorem]{Remark}

\newcommand{\eqnsection}{
\renewcommand{\theequation}{\thesection.\arabic{equation}}
    \makeatletter
    \csname  @addtoreset\endcsname{equation}{section}
    \makeatother}
\eqnsection

\def\E{\mathbb{E}}

\def\N{\mathbb{N}}

\def\R{\mathbb{R}}

\def\Pb{\mathbb{P}}

\renewcommand{\i}{\textbf{i}}

\newcommand{\equi}{\; \mathop{\sim}\limits}
\def\={{\,\;\mathop{=}\limits^{\text{(law)}}\;\,}}

\def\qed{\hfill$\square$}

\newcommand*\pFqskip{8mu}
\catcode`,\active
\newcommand*\pFq{\begingroup
        \catcode`\,\active
        \def ,{\mskip\pFqskip\relax}%
        \dopFq
}
\catcode`\,12
\def\dopFq#1#2#3#4#5{%
        {}_{#1}F_{#2}\biggl[\genfrac..{0pt}{}{#3}{#4};#5\biggr]%
        \endgroup
}

\allowdisplaybreaks

\makeatletter
\@namedef{subjclassname@2020}{%
  \textup{2020} Mathematics Subject Classification}
\makeatother

\begin{document}

\title[]{A pursuit problem for squared Bessel processes}
\author[Christophe Profeta]{Christophe Profeta}

\address{
Universit\'e Paris-Saclay, CNRS, Univ Evry, Laboratoire de Math\'ematiques et Mod\'elisation d'Evry, 91037, Evry-Courcouronnes, France.
 {\em Email}: {\tt christophe.profeta@univ-evry.fr}
  }



\begin{abstract} 
In this note, we are interested in the probability that two independent squared Bessel processes do not cross for a long time. We show that this probability has a power decay which is given by the first zero of some hypergeometric function. We also compute along the way the distribution of  the location where the crossing eventually occurs.
\end{abstract}

\maketitle

\section{Introduction}
\subsection{Statement of the main result}
Let $X$ and $Y$ be two independent squared Bessel processes of respective dimensions $\alpha>0$ and $\beta\geq0$. We denote by $\Pb_{(x,y)}$ the law of the pair $(X,Y)$ when starting from $(x,y)$ and we assume that $0\leq x < y$. In this note, we are interested in the probability that the process $X$ remains below $Y$ for a long time:
$$\Pb_{(x,y)}\left(\forall s\leq t,\, X_s < Y_s\right)\qquad \text{as } t\rightarrow +\infty.$$
This question is often labelled in the literature as a capture  or a pursuit problem. In most papers, the set-up is a Brownian one and 
the leading process is called a lamb or a prisoner, while its pursuers are wolves or policemen, see for instance \cite{Ken, LiSh, ReKr} and the references therein. One may also see this question as a persistence problem for the difference of two squared Bessel processes. Indeed, a natural way to study such  a question is to introduce the stopping time 
$$T= \inf\{t\geq0,\, X_t = Y_t\}$$
and study its tail asymptotics since $\Pb_{(x,y)}\left(\forall s\leq t,\, X_s < Y_s\right) = \Pb_{(x,y)}\left(T> t\right)$. Unlike the sum of squared Bessel processes, the difference is no longer Markovian and studying such asymptotics is thus more delicate. We refer to the two surveys \cite{AuSi, BMS} for an extensive background on persistence probabilities, as well as their links with many physic phenomenons.
When dealing with self-similar processes, one generally obtains a power-law decay. However, except in very specific situations, actually computing the exponent is usually complicated. The results in this paper are no exception: we will see that the persistence exponent of $X-Y$ is given by the first zero of a certain hypergeometric function.  \\

\noindent
In order to state our results, let us define the hypergeometric function ${}_2F_1$ (see for instance \cite[Chapter 9.1]{GrRy}): 
$$ \pFq{2}{1}{a\quad b}{c}{z} = \sum_{n=0}^{+\infty}  \frac{(a)_n(b)_n}{(c)_n} \frac{z^n}{n!}
$$
where $(a)_0=1$ and $(a)_n = a(a+1)\ldots (a+n-1)=\dfrac{\Gamma(a+n)}{\Gamma(a)}$ for $n\in \N$.
We denote by $F_{\alpha,\beta}$ the hypergeometric function
$$F_{\alpha, \beta}(s) =  \pFq{2}{1}{ \frac{\alpha+\beta}{2}-1+s \quad -s}{ \frac{\alpha}{2}}{\frac{1}{2}}, \qquad s\geq0.$$
We show below a typical plot of the function $F_{\alpha,\beta}$.

\begin{figure}[h!]
\includegraphics[height=5cm]{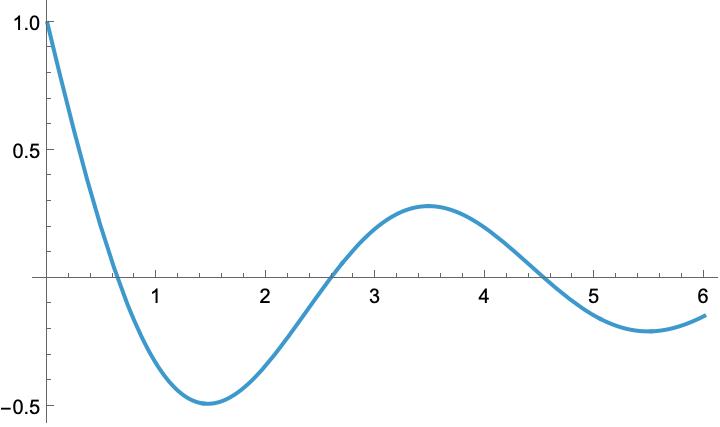}
\caption{Graph of $F_{5,3}(s)$ for $s\in[0,6]$}
\end{figure}

\noindent
Our first result gives the distribution of the position of $X$ and $Y$ when they cross each other.

\begin{proposition}\label{prop:X_T}
Let us denote by $\theta$ the first positive zero of the hypergeometric function $F_{\alpha,\beta}$: 
$$\theta = \inf\left\{s\geq0,\, F_{\alpha, \beta}(s)=0  \right\}.$$
Under $\Pb_{(0,y)}$ the stopping time $T$ is a.s. finite and the Mellin transform of $X_{T}$ is given by:
$$ \E_{(0,y)}\left[ X_{T}^{s} \right]= \frac{(2y)^s}{F_{\alpha,\beta}(s)  },\qquad s\in[0,\theta).$$
In particular, $\theta$ is equal to  $1$  when $\alpha=\beta$. Otherwise,  $\theta<1$ when $\alpha < \beta$, and  $\theta>1$ when $\alpha>\beta$.
\end{proposition}

It might be surprising to find that $T$ is always finite whatever the value of $\alpha>0$ and $\beta\geq0$, in particular even if $Y$ is transient (i.e. $\beta>2$). A possible interpretation is the following law of the iterated logarithm, see \cite{Shi}, which states that a squared Bessel process $Y$ of dimension $\beta>2$ satisfies: 
$$ \limsup_{t\rightarrow+\infty}  \frac{Y_t}{2 t \ln(\ln(t))} = 1\quad \text{a.s.} \quad \text{and}\quad  \limsup_{t\rightarrow+\infty}  \frac{\inf_{s\geq t}Y_s}{2 t \ln(\ln(t))} = 1\quad \text{a.s.}   $$
In particular, the limit sup of squared Bessel processes does not depend on the dimension. So in our set-up, even if $\beta$ is large, the fluctuations of the future infimum of the "upper" process $Y$ remain of the same order as those of the "lower" process $X$, which might explain why a crossing always occurs with probability one.\\

\noindent
The hypergeometric function appearing in Proposition \ref{prop:X_T} may be computed explicitly in several situations. 
\begin{enumerate}
\item For instance, when $\alpha=\beta$, we have from \cite[p.557, Formula 15.1.24]{AbSt}:
\begin{equation}\label{Fa=b}
F_{\alpha,\alpha}(s)=\pFq{2}{1}{\alpha-1+s\quad -s}{\frac{\alpha}{2}}{\frac{1}{2}}=  \sqrt{\pi} \frac{\Gamma(\frac{\alpha}{2})}{\Gamma\left(\frac{1-s}{2}\right)\Gamma\left(\frac{s+\alpha}{2}\right)},\qquad\text{and}\quad \theta =1.
\end{equation}
\item Another example is the case 
$\alpha+\beta=4$ for which we obtain from Bailey's formula \cite[p.557, Formula 15.1.26]{AbSt}:
$$F_{\alpha, 4-\alpha}(s)=  \pFq{2}{1}{ 1+s \quad -s}{ \frac{\alpha}{2}}{\frac{1}{2}} =2^{1-\frac{\alpha}{2}} \sqrt{\pi}\frac{ \Gamma\left(\frac{\alpha}{2}\right)}{ \Gamma\left(\frac{2+2s+\alpha}{4}\right) \Gamma\left(\frac{\alpha-2s}{4}\right)},\qquad\text{and}\quad \theta= \frac{\alpha}{2}. $$
\item As a last example, we assume that $\alpha=\beta+2$. From \cite[p.557, Formula 15.1.25]{AbSt}, we deduce that 
$$
F_{\beta+2, \beta}(s) = \frac{2\sqrt{\pi}}{\beta+2s} \Gamma\left(\frac{\beta}{2}+1\right)\left(\frac{1}{\Gamma\left(\frac{\beta+s}{2}\right) \Gamma\left(\frac{1-s}{2}\right)} - \frac{1}{\Gamma\left(\frac{\beta+s+1}{2}\right) \Gamma\left(-\frac{s}{2}\right)} \right).$$
In this case, $\theta$ does not seem to admit an explicit expression.
\end{enumerate}

\noindent
We now show that $\theta$ actually  controls the decay of $T$. The underlying idea, which was already successfully applied in \cite{PrSi, Pro}, is as follows. If $T$ were independent of $X$, then, by scaling, the random variables $X_T$ and  $T\times X_1$ would have the same distribution. Since $X_1$ admits moments of all orders, this would imply that the finiteness of the fractional moments of $X_T$ and $T$ should be equivalent. This heuristic result turns out to be correct, as evidenced by the following theorem.
\begin{theorem}\label{theo:1}
 For $\lambda\geq 0$ and $0\leq x<y$, we have the equivalence:
$$\E_{(x,y)}\left[T^{\lambda}\right] <+\infty \qquad  \Longleftrightarrow\qquad  \lambda < \theta.$$
In particular, when $\alpha>\beta$, it holds
$$\E_{(x,y)}\left[T\right]  = \frac{y-x}{\alpha-\beta}.$$
\end{theorem}

\begin{remark}\label{rem:Tab}
A consequence of this result is the monotonicity of $\theta$ in its arguments. Indeed,  let us emphasize the dependence in $\alpha$ and $\beta$ by writing $T=T_{\alpha,\beta}$. Using the comparison theorem for SDE \cite[Chapter IX, Theorem 3.8]{ReYo} and a coupling argument, it is immediate to see that for $\lambda>0$, the function
$(\alpha,\beta) \mapsto \E[T_{\alpha,\beta}^\lambda]$ is decreasing in $\alpha$ and increasing in $\beta$. As a consequence, the same is true for $\theta$, which is thus also non-decreasing in $\alpha$ and  non-increasing in $\beta$.
\end{remark}

\noindent
We finally give an asymptotics for the tail decay of $T$.  To this end, observe that from Lebedev \cite[Section 9.4]{Leb} the function $F_{\alpha,\beta}$ is an entire function. As a consequence,  there exists $m\in \N$ such that the following factorization holds
$$F_{\alpha,\beta}(s) = (s-\theta)^m G_{\alpha,\beta}(s)$$
where $G_{\alpha,\beta}$ is such that $G_{\alpha,\beta}(\theta)\neq0$.
\begin{corollary}\label{cor:1}
Let $\delta>0$.  There exist two positive constants $\kappa_1$ and $\kappa_2$  such that
$$\frac{\kappa_1}{t^{\theta+\delta}} \leq \Pb_{(0,y)}\left(T\geq t\right) \leq  \kappa_2\frac{(\ln(t))^m}{t^{\theta}}, \qquad \text{as }t\rightarrow +\infty.$$
\end{corollary}

When $\alpha=\beta=n\in \N$, it is well-known that there exist two independent $n$-dimensional Brownian motions $B$ and $W$ such that $\| B\|_2^2 = X$ and  $\| W\|_2^2 = Y$
where $\|\centerdot\|_2$ denotes the standard Euclidian norm. Therefore, Theorem  \ref{theo:1} and Corollary \ref{cor:1} state that the probability that $B$ remains closer to the origin that $W$  up to time $t$ decays essentially as a power $-1$, independently of the dimension. In dimension 1, this is an immediate consequence of the fact that  $S =(W+B)/\sqrt{2}$ and $D=(W-B)/\sqrt{2}$ are independent Brownian motions. Indeed, if we denote by  $\tau_a^Z= \inf\{t\geq 0, \; Z_t=a\}$ the first hitting time of  $a\in \R$ by a stochastic process $Z$, then
\begin{align*}
T &=   \inf\{t\geq 0, B_t^2 = W_t^2\}\\
    &=  \inf\{t\geq 0, (W_t+B_t)(W_t- B_t) =0\}  =   \inf\{t\geq 0, S_t\, D_t=0\}  = \tau_0^S \wedge \tau_0^D
    \end{align*}
hence, since $S_0=D_0=\sqrt{y/2}$ under $\Pb_{(0,y)}$,
$$ \Pb_{(0,y)}(T\geq t) =  \Pb_{(0,y)}(\tau_0^S \geq t, \tau_0^D \geq t)  = \left(\Pb_{(0,y)}(\tau_0^S \geq t)\right)^2 \equi_{t\rightarrow +\infty} 
\left(\sqrt{\frac{y}{2}} \sqrt{\frac{2}{\pi t}}\right)^2 = \frac{y}{\pi t}. $$

\begin{remark}
To complement Corollary \ref{cor:1}, we note that the short-time asymptotics of $T$ decays (at least) exponentially. Indeed, since the paths of $X$ and $Y$ are continuous, we have: 
$$\Pb_{(0,y)}(T\leq t) = \Pb_{(0,y)}\left(T\leq t, \tau_{\inf_{[0,t]} Y_s}^{X} \leq t \right) \leq\Pb_{(0,y)}\left( \tau_{\inf_{[0,t]} Y_s}^{X} \leq t\right).     $$
We then decompose this probability into two terms: 
\begin{multline*}
\Pb_{(0,y)}\left(  \tau_{\inf_{[0,t]} Y_s}^{X} \leq t,  \;  \inf_{[0,t]} Y_s > \frac{y}{2}\right) + \Pb_{(0,y)}\left( \tau_{\inf_{[0,t]} Y_s}^{X} \leq t,  \;  \inf_{[0,t]} Y_s \leq \frac{y}{2}\right) \\
\leq \Pb_{(0,y)}\left( \tau_{y/2}^{X} \leq t\right) + \Pb_{(0,y)}\left( \tau^Y_{y/2}  \leq t\right).
\end{multline*}
The result now follows from the fact that the short-time asymptotics of the hitting times of squared Bessel processes decay exponentially. Specifically, if $Z$  is a squared Bessel process with index $\nu$ started from $z\geq 0$,  the Laplace transform of $\tau_a^Z$ is given by
$$\E_{z}\left[e^{-\lambda \tau_a^{Z}}\right] = \begin{cases}
\dfrac{z^{-\nu/2} I_\nu (\sqrt{2\lambda z })}{a^{-\nu/2} I_\nu (\sqrt{2\lambda a})} & \qquad \text{if } z\leq a,\\
\dfrac{z^{-\nu/2} K_\nu (\sqrt{2\lambda z })}{a^{-\nu/2} K_\nu (\sqrt{2\lambda a})} & \qquad \text{if } z\geq a
\end{cases}$$
where $I_\nu$ and $K_\nu$ denote the usual modified Bessel functions. Using an integration by parts and letting $\lambda\rightarrow+\infty$, we deduce that
$$\int_0^{+\infty} e^{-\lambda t} \Pb_z\left(\tau_a^Z\leq t\right) dt \equi_{\lambda\rightarrow+\infty} \frac{z^{-\nu/2}}{a^{-\nu/2}}  \frac{1}{\lambda}e^{-\sqrt{2\lambda}|\sqrt{z}-\sqrt{a}| }. $$
Finally, applying the Tauberian theorem of exponential type (see \cite[Theorem 4.2.9]{BGT}), we conclude that
$$\lim_{t\downarrow0} t \ln\left(\Pb_z\left(\tau_a^Z\leq t\right)\right) = - \frac{(\sqrt{z}-\sqrt{a})^2}{2}.$$
\end{remark}

\subsection{The values of $\theta$}
When $\alpha=\beta$, we have already checked that $\theta=1$. We now prove that $\theta$ is smaller or greater than 1, according as whether $\alpha$ is  smaller or greater than $\beta$. 
\begin{enumerate}
\item Assume that $\alpha<\beta$. Then, taking $s=1$,  the hypergeometric function simplifies to 
$$F_{\alpha,\beta}(1)= \pFq{2}{1}{\frac{\alpha+\beta}{2}\quad -1}{\frac{\alpha}{2}}{\frac{1}{2}} = 1 -   \frac{1}{\alpha}\left(\frac{\alpha+\beta}{2}\right)<0$$
hence, by continuity, since $F_{\alpha,\beta}(0)=1$, its first zero $\theta$ lies in $(0,1)$.
\item Assume now that $\alpha>\beta$ and take $s\in(0,1)$. We have by definition, with $\gamma = \frac{\alpha+\beta}{2}-1$:
$$F_{\alpha,\beta}(s) =1-\frac{s(\gamma+s)}{\alpha}  \left(1+  \sum_{n=1}^{+\infty} \frac{(1-s)_n(\gamma+s+1)_n  }{(\alpha/2+1)_n}  \frac{1}{(n+1)!} \left(\frac{1}{2}\right)^n\right).$$
Now, on the one hand, if $\gamma + s\in(-1,0)$, then all the terms are positive and $F_{\alpha,\beta}(s)>0$. On the other hand, if  $\gamma+s\geq0$, by letting $\beta$ increase up to $\alpha$, we deduce that 
$$F_{\alpha,\beta}(s)\geq F_{\alpha,\alpha}(s)>0,$$
which is strictly positive from (\ref{Fa=b}) since $s\in(0,1)$. As a consequence, $\theta$ is strictly greater than one since $F_{\alpha,\beta}(1)>0$.\\ 
\end{enumerate}

We provide below a numerical simulation\footnote{The graph has been obtained using the \emph{FindRoot} and \emph{Hypergeometric2F1} functions in Mathematica.} of $(\alpha,\beta)\rightarrow \theta(\alpha, \beta)$ in which one can observe the monotonicity of $\theta$ in its arguments. Another interesting feature of this graph is that there appear to be different regimes as  $\alpha\downarrow0$, depending on whether $\beta>2$ or $\beta<2$.  We will discuss these limits in Section \ref{sec:limalpha}.

\begin{figure}[h!]\label{fig:theta}
\includegraphics[height=8cm]{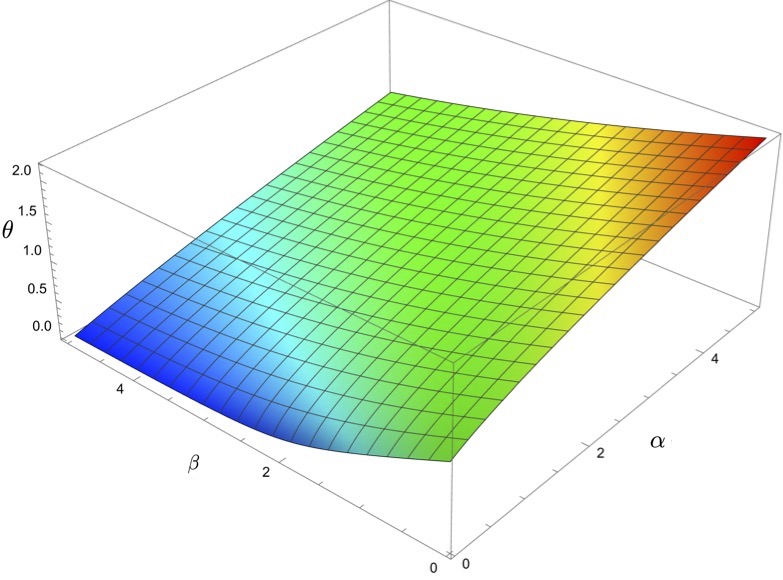}
\caption{Simulation of $\theta(\alpha, \beta)$ for $(\alpha, \beta) \in (0,5)^2$ }
\end{figure}

The remainder of the paper is organized as follows.  In Section \ref{sec:2}, we compute the Mellin transform of $X_T$ under $\Pb_{(0,y)}$.  In Section \ref{sec:3}, we show that the finiteness of the fractional moments of $T$ does not depend on the starting points $(x,y)$, allowing us to restrict our attention to the case $x=0$. The proof of Theorem \ref{theo:1} is then given in Section \ref{sec:4}, while that of Corollary \ref{cor:1} is provided in Section \ref{sec:5}. Finally, Section \ref{sec:limalpha} studies the limit of $\theta(\alpha,\beta)$ as $\alpha\downarrow0$.

\section{Computation of the Mellin transform of $X_T$}\label{sec:2}
We start by proving Proposition \ref{prop:X_T}. Recall that under $\Pb_{(x, y)}$ the pair $(X, Y)$ is a solution of the SDE, see \cite[Chapter XI]{ReYo}:
\begin{equation}\label{eq:SDE}
X_t = x +2 \int_0^{t} \sqrt{X_s} dB_s + \alpha t\qquad \text{and }\qquad Y_t = y +2 \int_0^{t} \sqrt{Y_s} dW_s + \beta t,\quad \qquad t\geq0,
\end{equation}
where $B$ and $W$ are two independent Brownian motions.  
In the following, to simplify the notation, we shall write for a random variable $Z$  and for $\gamma \in \R$:
$$\E[Z_+^\gamma]= \E[Z^\gamma 1_{\{Z>0\}}].$$
Take $r,\nu\in(0,1)$ such that $r+\nu>1$ and $\nu < \frac{\alpha+\beta}{2}$. 
Applying the Markov property and the scaling property of squared Bessel processes, we have 
\begin{align}
\notag \int_0^{+\infty} t^{-r} \E_{(0,y)}\left[(X_t - Y_t)_+^{-\nu}\right] dt &= \int_0^{+\infty} \E_{(0,y)}\left[(u+T)^{-r} \E_{(X_T, X_T)}\left[  \left(X_u - Y_u\right)_+^{-\nu}\right]1_{\{T<+\infty\}} \right] du\\
\label{eq:markov}&= \int_0^{+\infty} \E_{(0,y)}\left[ X_T^{1-\nu}    (sX_T+T)^{-r}1_{\{T<+\infty\}} \right]\E_{(1,1)}\left[  \left(X_s - Y_s\right)_+^{-\nu} \right] ds.
\end{align}
Note that we cannot just take $r=0$, as both sides would be infinite. Indeed, in this case, for the left-hand side to be finite, the integral in $t$ requires that $\nu>1$ while the expectation requires that $\nu<1$.  We shall thus evaluate semi-explicitly both sides, show that the poles at $r+\nu=1$ cancel, and then use analytic continuation.
To compute the positive parts in (\ref{eq:markov}), recall the formula for any positive r.v. $Z$ which admits some negative moments, see \cite[Lemma 1]{PrSi}:
$$
\E\left[ Z_+^{-\nu}\right]=\frac{\Gamma(1-\nu)}{\pi}\int_0^{+\infty}  \lambda^{\nu-1}\E\left[\sin\left(\lambda Z+ \frac{\pi}{2}\nu\right)\right]d\lambda 
.$$
From \cite[Chapter XI, p.441]{ReYo} and analytic continuation, the Fourier transform of $X-Y$ is given,  since $X$ and $Y$ are independent, by 
$$\E_{(x,y)}\left[e^{\i \lambda (X_t - Y_t)}\right] =(1-2\i \lambda t)^{-\alpha/2}(1+2\i \lambda t)^{-\beta/2}  \exp\left(\frac{\i \lambda (x-y) -2\lambda^2 t (x+y)}{1+4\lambda^2 t^2}\right)$$
hence, after a change of variable
\begin{multline}\label{eq:X+}
\frac{\pi}{\Gamma(1-\nu)}\E_{(x,y)}\left[ \left(X_t - Y_t\right)_+^{-\nu}\right] \\ =     t^{-\nu} \Im\left( e^{\i \frac{\pi}{2}\nu}\int_0^{+\infty}  \xi^{\nu-1} (1-2\i \xi)^{-\alpha/2}(1+2\i \xi)^{-\beta/2}  \exp\left(\frac{\i \xi (x-y) -2\xi^2 (x+y)}{ t(1+4\xi^2)}\right)  d\xi \right) $$
\end{multline}
where $\Im$ denotes the imaginary part.
On the one hand, taking $x=0$ in (\ref{eq:X+}) and plugging this expression in (\ref{eq:markov}), we deduce thanks to Fubini's theorem that the left-hand side of (\ref{eq:markov}) equals 
$$\frac{\Gamma(1-\nu)\Gamma(r+\nu-1)}{\pi y^{r+\nu-1}}  \Im\left(    e^{\i \frac{\pi}{2}(1-r)} \int_0^{+\infty}  \xi^{-r} (1-2\i \xi)^{- \frac{\alpha}{2}}(1+2\i \xi)^{ r+\nu-1-\frac{\beta}{2}} \right)d\xi.$$
Similarly, plugging (\ref{eq:X+}) with $x=y=1$ in the right-hand side of (\ref{eq:markov}) and applying Fubini's theorem, we obtain
$$\frac{\Gamma(1-\nu)\Gamma(r+\nu-1)}{\pi }    \Im\left(e^{\i \frac{\pi}{2}\nu} \int_0^{+\infty}  \xi^{\nu-1} (1-2\i \xi)^{-\alpha/2}(1+2\i \xi)^{-\beta/2}  G_{r,\nu}(\xi) d\xi \right)$$
where $G_{r,\nu}$ is defined by 
$$G_{r,\nu}(\xi) := \frac{1}{\Gamma(r+\nu-1)}  \E_{(0,y)}\left[ X_T^{1-\nu} \int_0^{+\infty} (uX_T+T)^{-r} u^{-\nu}  e^{-\frac{4 \xi^2}{u(1+4\xi^2)}} du\,1_{\{T<+\infty\}}\right].$$
Setting  $\displaystyle \eta(\xi) = \frac{4 \xi^2}{1+4\xi^2}$,  using the change of variables $u=1/s$ and an integration by parts, the integral reads: 
 \begin{align*}
 \int_0^{+\infty} (uX_T+T)^{-r} u^{-\nu}  e^{-\eta(\xi)/u} du&=   \int_0^{+\infty} (X_T+Ts)^{-r} s^{\nu+r-2}  e^{-s \eta(\xi) }ds \\
 &= \frac{1}{r+\nu-1} \int_0^{+\infty} (X_T+Ts)^{-r} s^{\nu+r-1}  e^{-s \eta(\xi) } \left(r\frac{T}{X_T+Ts}+\eta(\xi) \right)  ds. 
\end{align*}
Note that in the integration by parts, the boundary parts cancel since $\nu+r>1$. As a consequence, using the functional equation of the Gamma function, $G_{r,\nu}$ equals 
$$
G_{r,\nu}(\xi) = \frac{1}{\Gamma(\nu+r)}\E_{(0,y)}\left[X_T^{1-\nu} \int_0^{+\infty} s^{\nu+r-1} (X_T+T s)^{-r} e^{-\eta(\xi) s}\left(r\frac{T}{X_T+Ts}+\eta(\xi) \right) ds\,1_{\{T<+\infty\}}\right].
$$
Simplifying the Gamma terms, Equation (\ref{eq:markov})  now reads 
\begin{multline}\label{eq:I=I}
\frac{1}{y^{r+\nu-1}}\Im\left(    e^{\i \frac{\pi}{2}(1-r)}\int_0^{+\infty}\xi^{-r}(1-2\i \xi)^{-\frac{\alpha}{2}}(1+2\i \xi)^{r+\nu-1-\frac{\beta}{2}} d\xi \right) \\=  \Im\left(e^{\i \frac{\pi}{2}\nu} \int_0^{+\infty}  \xi^{\nu-1} (1-2\i \xi)^{-\frac{\alpha}{2}}(1+2\i \xi)^{-\frac{\beta}{2}}  G_{r,\nu}(\xi) d\xi \right)
\end{multline}
and this formula extends by analyticity to $r+\nu>0$.
By monotone convergence, we then deduce that
$$G_{r,\nu}(\xi) \xrightarrow[r\downarrow0]{} \frac{1}{\Gamma(\nu)}\E_{(0,y)}\left[X_T^{1-\nu} \int_0^{+\infty} s^{\nu-1} e^{-\eta(\xi) s} \eta(\xi)ds \, 1_{\{T<+\infty\}}\right]=\E_{(0,y)}\left[X_T^{1-\nu}\,1_{\{T<+\infty\}}\right]  \left(\eta(\xi)  \right)^{1-\nu}.  $$
As a consequence, letting $r\downarrow 0$ in (\ref{eq:I=I}), we obtain the formula
\begin{multline}\label{eq:gauchedroite}
\frac{1}{y^{\nu-1}}\Im\left(  \i \int_0^{+\infty}(1-2\i \xi)^{-\frac{\alpha}{2}}(1+2\i \xi)^{\nu-1-\frac{\beta}{2}} d\xi \right) \\
= \E_{(0,y)}\left[X_T^{1-\nu}1_{\{T<+\infty\}} \right]  \Im\left( e^{\i \frac{\pi}{2}\nu}  \int_0^{+\infty}  \xi^{1-\nu} (1-2\i \xi)^{\nu-1-\alpha/2}(1+2\i \xi)^{\nu-1-\beta/2} d\xi\right)
\end{multline}
and it remains to compute both integrals. Let us start with the left-hand side. Taking the imaginary part in Lemma \ref{lem:app} from the Appendix, we have
\begin{multline*}
\Im\left(\i\int_0^{+\infty} \xi^{\gamma-1} (1-2\i \xi)^{-\lambda} (1+2\i\xi )^{-\mu} d\xi\right)\\
=\frac{\Gamma(\gamma) \Gamma(\mu+\lambda-\gamma)}{\Gamma(\lambda)\Gamma(\mu)} 2^{-\gamma}  \left( \cos\left(\frac{\pi}{2}\gamma\right)  B\left(\lambda, 1-\gamma  \right)2^{-\lambda} \pFq{2}{1}{1-\mu \quad\lambda}{\lambda -\gamma + 1 }{\frac{1}{2}} \right.\\
\left.+\cos\left(\frac{\pi}{2}\gamma\right) B\left( \mu, 1-\gamma \right)  2^{-\mu} \pFq{2}{1}{1-\lambda \quad\mu}{\mu -\gamma + 1 }{\frac{1}{2}}\right)
\end{multline*}
where $B(x,y)$ denotes the usual Beta function. Letting the parameter $\gamma\uparrow 1$ and using the asymptotics $B(x, y) \equi_{y\rightarrow 0} \frac{1}{y}$ together with the identity $\displaystyle \pFq{2}{1}{a \quad b}{b }{\frac{1}{2}}=2^{-a}$, we thus deduce that
\begin{equation}\label{eq:gauche}
\Im\left(\i\int_0^{+\infty}  (1-2\i \xi)^{-\frac{\alpha}{2}}(1+2\i \xi)^{\nu-1-\frac{\beta}{2}} d\xi\right) =\frac{ \pi \Gamma\left(\frac{\alpha+\beta}{2}-\nu\right)}{\Gamma(\frac{\alpha}{2})\Gamma(\frac{\beta}{2}+1-\nu)} 2^{\nu - \frac{\alpha+\beta}{2}-1}.
\end{equation}
On the other side,  multiplying Lemma \ref{lem:app} in the Appendix by $e^{\i \frac{\pi}{2}\nu}$ and taking the imaginary part, we observe that only the first term will remain, i.e.
\begin{multline}\label{eq:droite}
\Im\left( e^{\i \frac{\pi}{2}\nu}  \int_0^{+\infty}  \xi^{1-\nu} (1-2\i \xi)^{\nu-1-\alpha/2}(1+2\i \xi)^{\nu-1-\beta/2} d\xi\right)
\\=- \frac{\Gamma(2-\nu)\Gamma\left(\frac{\alpha+\beta}{2}-\nu\right)}{\Gamma\left(\frac{\beta}{2}+1-\nu\right)}
2^{\nu-2}\sin(\pi \nu) \frac{\Gamma(\nu-1)}{\Gamma\left(\frac{\alpha}{2}\right)} 2^{\nu-1-\frac{\alpha}{2}}  \pFq{2}{1}{\nu-\frac{\beta}{2}\quad \frac{\alpha}{2}+1-\nu}{ \frac{\alpha}{2}}{\frac{1}{2}}.
\end{multline}
Plugging (\ref{eq:gauche}) and (\ref{eq:droite}) in (\ref{eq:gauchedroite}) and using the reflection formula for the Gamma function, i.e. $ \Gamma(2-\nu) \sin(\pi \nu)\Gamma(\nu-1)=-\pi$ for $\nu\in(0,1)$, we finally conclude that  
\begin{equation}\label{eq:Xnu}
 \E_{(0,y)}\left[ X_T^{1-\nu} 1_{\{T<+\infty\}}\right]=\left(2 y\right)^{1-\nu}  \frac{2^{ 1- \frac{\beta}{2}}}{ \pFq{2}{1}{\nu-\frac{\beta}{2}\quad \frac{\alpha}{2}+1-\nu}{ \frac{\alpha}{2}}{\frac{1}{2}}}
 \end{equation}
and the result follows from Euler's transformation 
$$ \pFq{2}{1}{\nu-\frac{\beta}{2}\quad \frac{\alpha}{2}+1-\nu}{ \frac{\alpha}{2}}{\frac{1}{2}}= 2^{1-\frac{\beta}{2}}  \pFq{2}{1}{\frac{\alpha+\beta}{2}-\nu\quad\nu-1}{ \frac{\alpha}{2}}{\frac{1}{2}} = 2^{1-\frac{\beta}{2}}F_{\alpha,\beta}(1-\nu).  $$
In particular, letting $\nu\uparrow1$ in (\ref{eq:Xnu}), we deduce that  $\Pb_{(0,y)}(T<+\infty) = 1$, i.e., that $T$ is a.s. finite. \qed

\section{Some first estimates}\label{sec:3}

\subsection{From $\Pb_{(x,y)}$ to $\Pb_{(0,1)}$}
Before tackling the proof of Theorem \ref{theo:1}, we gather here some preliminary computations.
We start by proving that the finiteness of  $\E_{(x,y)}[T^\lambda]$ is equivalent to that of $\E_{(0,1)}[T^\lambda]$. Working with $x=0$ will be key, as we aim to compare the fractional moments of $T$ with those of $X_T$  given in Proposition \ref{prop:X_T}.
\begin{lemma}\label{lem:Txy}
Let $0\leq x<y$ be fixed. There exist two positive constants $\kappa_1, \kappa_2$ such that for $0\leq \lambda \leq \theta$,
$$ \kappa_1 \E_{(0,1)}\left[T^\lambda\right] -\kappa_2 \leq \E_{(x,y)}\left[T^\lambda\right] \leq  y^{\lambda}\E_{(0,1)}\left[T^\lambda\right]. $$
\end{lemma}
\begin{proof}
Notice first that by a coupling argument  $\Pb_{(x,y)}(T\geq t) \leq \Pb_{(0,y)}\left(T \geq t\right)$ which implies by scaling that for $\lambda\geq0$:
$$
\E_{(x,y)}\left[T^\lambda\right] \leq \E_{(0,y)}\left[T^\lambda\right] = y^\lambda \E_{(0,1)}\left[T^\lambda\right]. 
$$
To prove a converse inequality, take $a\in(0,y)$ and recall that $\tau_y^Y = \inf\{t\geq0,\; Y_t =y\}$ denotes the first hitting time of $y$ by the process $Y$. Then, applying the strong Markov property at the time $\tau_y^Y$, 
\begin{align*}
\Pb_{(0,a)}(T>t) &= \Pb_{(0,a)}(T>t, T\leq \tau_y^Y) +\Pb_{(0,a)}(T>t, T> \tau_y^Y)\\
&\leq 2 \Pb_{(0,a)}(\tau_y^Y >t) + \E_{(0,a)}\left[ \Pb_{(X_{\tau_y^Y}, y)}( T>t-s)_{|s=\tau_y^Y} 1_{\{T>\tau_y^Y\}}1_{\{\tau_y^Y\leq t\}}\right]
\end{align*}
hence, for $\lambda\leq \theta$, using the standard inequality $(a+b)^\lambda \leq 2^\lambda (a^\lambda + b^\lambda)$,
\begin{align}
\notag\E_{(0,a)}[T^\lambda] & \leq 2\E_{(0,a)}[(\tau_y^Y)^\lambda] +  \int_0^{y} \int_0^{+\infty}  \E_{(z,y)}[(T+s)^{\lambda}] \Pb_{(0,a)}(X_{\tau_y^Y}\in dz, \tau_y^Y\in ds )\\
\notag&\leq2\E_{(0,a)}[(\tau_y^Y)^\lambda] + 2^{\lambda} \int_0^{y}  \int_0^{+\infty}  \left(\E_{(z,y)}[T^\lambda]+ s^\lambda\right)\Pb_{(0,a)}(X_{\tau_y^Y}\in dz, \tau_y^Y\in ds )\\
\label{eq:2theta}&\leq (2+2^\theta) \E_{(0,a)}[(\tau_y^Y)^\lambda] +  2^{\theta} \int_0^{y}  \E_{(z,y)}[T^\lambda]\Pb_{(0,a)}(X_{\tau_y^Y}\in dz).
\end{align}
Recall then that, since $Y_0=a<y$, the r.v. $\tau_y^Y$ admits positive moments of all orders, see for instance \cite{Shi2}, so the first term on the right-hand side is always finite. Now let $\delta \in(0,y)$. Since the function $z\rightarrow \E_{(z,y)}[T^\lambda]$ is decreasing on $[0,y]$, we have
$$ \int_0^{y}  \E_{(z,y)}[T^\lambda]\Pb_{(0,a)}(X_{\tau_y^Y}\in dz) \leq     \E_{(0,y)}[T^\lambda] \Pb_{(0,a)}(X_{\tau_y^Y}\leq \delta) +   \E_{(\delta,y)}[T^\lambda]. $$
As a consequence, plugging this last inequality in (\ref{eq:2theta}) and using the scaling property, we obtain the bound
\begin{equation}\label{eq:xy2}
\left(a^\lambda-  2^\theta y^\lambda  \Pb_{(0,a)}(X_{\tau_y^Y}\leq \delta)   \right)\E_{(0,1)}[T^\lambda] \leq(2+2^\theta) \E_{(0,a)}[(\tau_y^Y)^\lambda] +2^\theta \E_{(\delta,y)}[T^\lambda]
\end{equation}
in which we shall take $\delta$ small enough for the left-hand side to be strictly positive. Finally, it remains to write, applying the Markov property,
$$
 \E_{(x,y)}[(T-1)^\lambda  1_{\{T>1\}}]\geq \E\left[1_{\{T>1, Y_1>y, X_1\leq \delta  \} }\E_{(X_1, Y_1)}[T^\lambda]\right] \geq \E\left[1_{\{T>1, Y_1>y, X_1\leq \delta  \}}\right] \E_{(\delta, y)}[T^\lambda] 
$$
hence, going back to (\ref{eq:xy2}), we have thus proven that for $0\leq \lambda \leq \theta$, 
$$\kappa_1  \E_{(0,1)}\left[T^\lambda\right] \leq \kappa_2 +  \E_{(x,y)}[(T-1)^\lambda1_{\{T>1\}} ] \leq \kappa_2 +   \E_{(x,y)}[T^\lambda] $$
where $\kappa_1, \kappa_2$ are two positive constants. This is the lower bound of Lemma \ref{lem:Txy}.
\end{proof}

\subsection{Maximal inequalities }

To go from $X_T$ to $T$, we shall rely on maximal inequalities for squared Bessel processes. Indeed, recall from DeBlassie \cite{DeB} that there exist two constants  $c_{\alpha,\lambda}$ and $C_{\alpha,\lambda}$ such that for any stopping time $\zeta$ (with respect to $X$):
\begin{equation}\label{eq:BDG}
c_{\alpha,\lambda} \E_{(x,y)}[(\zeta+x)^\lambda]\leq  \E_{(x,y)}\left[\sup_{s\leq \zeta} X_s^\lambda\right] \leq C_{\alpha,\lambda} \E_{(x,y)}[(\zeta+x)^\lambda].
\end{equation}
As a consequence, taking $\zeta=T$, we first deduce that
\begin{equation}\label{eq:XTT}
\E_{(x,y)}\left[X_T^\lambda\right]\leq  \E_{(x,y)}\left[\sup_{s\leq T} X_s^\lambda\right]  \leq C_{\alpha, \lambda} \E_{(x,t)}[(T+x)^\lambda]
\end{equation}
which  implies that  $ \E_{(x,y)}\left[ T^\lambda\right] =+\infty$ as soon as $\lambda\geq \theta$. 
Note that a converse inequality exists: from Pedersen \cite{Ped}, if $\lambda> 1-\alpha/2$ and $\zeta$ is a stopping time such that $\E_{(x,y)}[\zeta^{\lambda}]<+\infty$, then
\begin{equation}\label{eq:Ped}
\E_{(x,y)}\left[\sup_{s\leq \zeta} X_s^{\lambda}\right] \leq \left(\frac{2\lambda}{2\lambda - (2-\alpha)} \right)^{\frac{2\lambda}{2-\alpha}} \E_{(x,y)}\left[X_{\zeta}^{\lambda}\right]. 
\end{equation}
Unfortunately, to apply this inequality in our set-up, one needs the assumption that $T$ admits moments of order up to $\theta$, which is precisely what we wish to prove.

\section{Proof of Theorem \ref{theo:1}}

We prove in this section that $\E_{(0,y)}[T^\lambda]$ is finite for $0\leq \lambda<\theta$. As explained in the first section, the idea is to compare the moments of $T$ with those of $X_T$. To do so, we  construct martingales involving  both $T$ and $X_T$.  Due to a technical difficulty, we will need to separate the cases $\alpha\leq \beta$ and  $\alpha>\beta$, constructing two different martingales accordingly.

\subsection{The case $\alpha\leq \beta$}\label{sec:4}

Recall that in this case $\theta\in(0,1]$. We first show that $\theta> 1-\frac{\beta}{2}$. Indeed, since the paths of $X$ and $Y$ are continuous and non-negative, the path inequality $T \leq \tau_0^Y$ implies that $\E_{(0,1)}[T^\lambda] \leq \E_{(0,1)}[(\tau_0^Y)^\lambda]$. The right-hand side is finite if and only if $\beta<2$ and $\lambda < 1-\frac{\beta}{2}$, see for instance \cite{Har}. As a consequence, we deduce from (\ref{eq:XTT}) the lower bound  $\theta> 1-\frac{\beta}{2}$, with a strict inequality since $F_{\alpha,\beta}\left(1-\frac{\beta}{2}\right) = 2^{\frac{\beta}{2}-1}>0$. \\

\noindent
We  now construct a  martingale. Let us set for $z>0$ 
\begin{equation}\label{eq:Psi}
\Psi^{(\beta)}_\lambda(z):= z^{\frac{\eta}{2}} K_{\eta}\left(\sqrt{2\lambda z}\right) = 2^{\frac{\eta}{2}-1} \lambda^{\frac{\eta}{2}} \int_0^{+\infty} e^{-\lambda s} e^{-\frac{z}{2s}} s^{\eta-1} ds
\end{equation}
where $\eta= 1-\frac{\beta}{2}\in(-\infty,1)$ and  $K_\eta$ denotes the usual McDonald's function.  In particular, $\Psi^{(\beta)}_\lambda$ is a solution of the ordinary differential equation:
\begin{equation}\label{eq:f''}
z f^{\prime\prime}(z) + \frac{\beta}{2} f^\prime(z) - \frac{\lambda}{2}f(z) = 0
\end{equation} such that $\displaystyle \lim_{z\rightarrow +\infty}\Psi^{(\beta)}_\lambda(z) =0$. 
Take $Y_0=y>0$ and $0<\varepsilon< y$.
Applying It\^o's formula, we deduce from (\ref{eq:f''}) that the process 
$$M_t = e^{-\lambda (t\wedge \tau_\varepsilon^Y)} \Psi_\lambda^{(\beta)}\left( Y_{ t\wedge \tau_\varepsilon^Y}\right), \quad t\geq0 $$
is a bounded local martingale, hence a true martingale under $\Pb_{(0,y)}$. 
Then Doob's optional sampling theorem yields the identity 
\begin{equation}\label{eq:Doob-}
\E_{(0,y)}\left[e^{-\lambda (T\wedge \tau_\varepsilon^Y)} \Psi_\lambda^{(\beta)}\left( Y_{ T\wedge \tau_\varepsilon^Y}\right)\right] = \Psi_\lambda^{(\beta)}(y).
\end{equation}
\subsubsection{We first assume that $\eta<0$.} In this case, we may let $\lambda\downarrow 0$ in (\ref{eq:Doob-}) and apply the monotone convergence theorem to obtain
\begin{equation}\label{eq:eta<0}
 \int_0^{+\infty} s^{\eta-1}   \E_{(0,y)}\left[e^{-\frac{Y_{ T\wedge \tau_\varepsilon^Y} }{2s}}\right] ds =\int_0^{+\infty} s^{\eta-1}  e^{-\frac{y}{2s}} ds 
 \end{equation}
i.e. $\E_{(0,y)}\left[Y_{ T\wedge \tau_\varepsilon^Y}^\eta\right] = y^\eta$. 
Then, subtracting $\E_{(0,y)}\left[\Psi_\lambda^{(\beta)}\left( Y_{ T\wedge \tau_\varepsilon^Y}\right)\right]$ on both sides of (\ref{eq:Doob-}), we obtain from (\ref{eq:Psi}) and (\ref{eq:eta<0}), using an integration by parts:
\begin{align*}
 & \E_{(0,y)}\left[ \int_0^{+\infty}    e^{-\lambda t} 1_{\{T\wedge \tau_\varepsilon^Y>  t\}} dt  \int_0^{+\infty} e^{-\lambda s}   e^{-\frac{Y_{ T\wedge \tau_\varepsilon^Y} }{2s}} s^{\eta-1} ds\right]\\       & \qquad\qquad = \frac{1}{\lambda} \int_0^{+\infty} e^{-\lambda s}  \  \E_{(0,y)}\left[e^{-\frac{Y_{ T\wedge \tau_\varepsilon^Y} }{2s}}-e^{-\frac{y}{2s}} \right] s^{\eta-1} ds\\
  &\qquad \qquad =  \int_0^{+\infty} e^{-\lambda s} \left( \int_s^{+\infty}   \E_{(0,y)}\left[e^{-\frac{y}{2u}} -e^{-\frac{Y_{ T\wedge \tau_\varepsilon^Y} }{2u}} \right]u^{\eta-1} du\right) ds.
  \end{align*}
Inverting these Laplace transforms yields the identity:
\begin{equation}\label{eq:defH}
  \int_0^{t}  s^{\eta-1}\E_{(0,y)}\left[ 1_{\{ T\wedge \tau_\varepsilon^Y>t-s\}} e^{-\frac{Y_{ T\wedge \tau_\varepsilon^Y}}{2s}}\right] ds = \int_t^{+\infty} s^{\eta-1} \E_{(0,y)}\left[e^{-\frac{y}{2s}} -e^{-\frac{Y_{ T\wedge \tau_\varepsilon^Y}}{2s}}   \right]ds=:H_{\varepsilon}(t).
  \end{equation}
We now compute the Mellin transform of $H_{\varepsilon}$. Recall the formula, for $a,b>0$ and $\nu\in(0,1)$:
\begin{equation}\label{eq:formulelog}
\int_0^{+\infty} t^{\nu-1} \left(e^{-\frac{a}{t}} - e^{-\frac{b}{t}}\right)dt = \Gamma(-\nu) \left(a^\nu - b^\nu\right). 
\end{equation}
Take $\lambda \in(0,\theta)$. Separating the cases $Y_{T\wedge \tau_\varepsilon^Y} \leq y$ and $Y_{T\wedge \tau_\varepsilon^Y} > y$ and using a change of variables, we may apply the Fubini-Tonelli theorem to obtain 
\begin{align}
\notag \int_0^{+\infty} t^{\lambda-\eta-1} H_\varepsilon(t) dt &= \int_0^{+\infty} t^{\lambda-1} \int_1^{+\infty} u^{\eta-1}  \E_{(0,y)}\left[e^{-\frac{y}{2ut}}-e^{-\frac{Y_{ T\wedge \tau_\varepsilon^Y} }{2ut}} \right] du dt\\
\notag &= \Gamma(-\lambda)  \int_1^{+\infty} u^{\eta-1}  \E_{(0,y)}\left[\left(\frac{y }{2u}\right)^\lambda- \left(\frac{Y_{ T\wedge \tau_\varepsilon^Y} }{2u}  \right)^\lambda \right] du\\
\label{eq:MelH}&=\Gamma(-\lambda)  \frac{2^{-\lambda}}{\eta-\lambda} \left( \E_{(0,y)}\left[Y_{ T\wedge \tau_\varepsilon^Y}^\lambda\right] - y^\lambda\right).
\end{align}
\noindent
We now come back to (\ref{eq:defH}) and observe that:
\begin{align*}
\notag H_\varepsilon(t) &\geq  \E_{(0,y)}\left[1_{\{ T\wedge \tau_\varepsilon^Y>t\}}   \int_0^{t} s^{\eta-1} e^{-\frac{Y_{ T\wedge \tau_\varepsilon^Y}}{2s}}ds  \, 1_{\{ Y_{T\wedge \tau_\varepsilon}\leq  t\}}\right]\\
\notag &\geq  \Pb_{(0,y)}\left( T\wedge \tau_\varepsilon^Y>t,  Y_{T\wedge \tau_\varepsilon^Y}\leq  t\right)  \int_0^{t}  s^{\eta-1} e^{-\frac{ t}{2s}} ds\\
&\geq  t^{\eta}   \left( \Pb_{(0,y)}\left( T\wedge \tau_\varepsilon^Y>t\right)  -   \Pb_{(0,y)}\left(  Y_{T\wedge \tau_\varepsilon^Y}>  t\right)\right)     \int_0^{1}  u^{\eta-1} e^{-\frac{1 }{2u}} du. 
\end{align*}
Integrating this relation against $(\lambda-\eta) t^{\lambda-\eta-1}$ on $(0, +\infty)$ for $\lambda\in(0,\theta)$, we deduce from (\ref{eq:MelH}) that  there exists a constant $K$ such that
\begin{align}
\notag \E_{(0,y)}\left[\left(T\wedge \tau_\varepsilon^Y\right)^\lambda\right]  &\leq  \E_{(0,y)}\left[Y_{T\wedge \tau_\varepsilon^Y}^{\lambda}\right] \left(1-  \frac{2^{-\lambda}K}{\lambda-\eta} \Gamma(-\lambda)  \right)  + \frac{2^{-\lambda}K}{\lambda-\eta} \Gamma(-\lambda)  y^{\lambda}  \\
\label{eq:T<Y} &\leq  \left(\E_{(0,y)}\left[Y_{T}^{\lambda}\right] + \varepsilon^\lambda\right) \left(1-  \frac{2^{-\lambda}K}{\lambda-\eta} \Gamma(-\lambda)  \right)  + \frac{2^{-\lambda}K}{\lambda-\eta} \Gamma(-\lambda)  y^{\lambda} .
\end{align}
The upper bound of Theorem \ref{theo:1} now follows by letting $\varepsilon\downarrow0$ and applying the monotone convergence theorem. \qed

\subsubsection{We now assume that $\eta\in(0,1)$.} Then, the function $\Psi_{\lambda}^{(\beta)}$ is well-defined and finite at $z=0$. As a consequence, we may take directly $\varepsilon=0$ (i.e. remove the stopping time $\tau_{\varepsilon}^Y$) when applying Doob's optional stopping theorem. However, we cannot let $\lambda\downarrow0$ in (\ref{eq:Doob-}) as both sides of (\ref{eq:eta<0}) would be infinite.
Nevertheless,  from Proposition \ref{prop:X_T} and (\ref{eq:formulelog}), 
$$ \int_0^{+\infty} s^{\eta-1} \left(  \E_{(0,y)}\left[e^{-\frac{Y_{ T} }{2s}}\right] -e^{-\frac{y}{2s}}\right) ds = \Gamma(-\eta)\left( \E_{(0,y)}\left[Y_{ T}^\eta\right] - y^\eta  \right) =0$$
and the previous argument applies mutadis mutandis, taking $\lambda \in (\eta, \theta)$.

\subsubsection{We now assume that $\eta=0$, (i.e. $\beta=2$).} To simplify, we will rely here on a coupling argument. Take $\varepsilon>0$ and recall from Remark \ref{rem:Tab} that for $\lambda >0$:
$$\E_{(0,y)}\left[T_{\alpha, 2}^\lambda\right] \leq \E_{(0,y)}\left[T_{\alpha, 2+\varepsilon}^\lambda\right]. $$
Consequently, we deduce from the first part of the proof that $\E_{(0,y)}\left[T_{\alpha, 2}^\lambda\right] <+\infty$ for any $\lambda<\theta(\alpha, 2+\varepsilon)$. The result then follows by letting $\varepsilon\downarrow0$, since $\theta$ is continuous in its parameters.\qed

\noindent
\begin{remark}\label{rem:T1}
Note that this argument no longer holds for $\alpha>\beta$. Indeed, in this latter case $\theta>1$, but since the right-hand side of (\ref{eq:T<Y}) has a pole at $\lambda=1$ (coming from the Mellin transform of $ H_{\varepsilon}$), we can only deduce that  $\E_{(0,y)}[T^\lambda]<+\infty$ for $\lambda \in (0,1)$.
\end{remark}

\subsection{The case $\alpha>\beta$}\label{sec:3.4}

Let us define
$$\gamma = \sup\{s\geq0,\; \E_{(0,y)}[T^s]<+\infty\}.$$
We already know from (\ref{eq:XTT}) that $\gamma\leq \theta$ so let us assume that $\gamma<\theta$. We shall prove by contradiction that this is not possible.  

\subsubsection{First step} We first prove that if $\gamma<\theta$, then  $\E_{(0,y)}[T^\gamma]<+\infty$, i.e. that the supremum in the definition of $\gamma$ is attained. To do so, observe first that as explained in Remark \ref{rem:T1}, $\gamma \geq 1$. This may also be checked using the following argument:  since $\alpha>\beta$, the SDE (\ref{eq:SDE}) defining $X$ and $Y$ implies that the process 
$$M_{t\wedge T} =Y_{t\wedge T}-X_{t\wedge T} +(\alpha-\beta)t\wedge T = y+ \int_0^{t\wedge T} \sqrt{Y_s} dW_s - \int_0^{t\wedge T} \sqrt{X_s} dB_s , \qquad t\geq0$$
is a positive local martingale, hence a supermartingale.
Applying Fatou's lemma, we thus have
$$(\alpha-\beta)\E_{(0,y)}[T] = \E_{(0,y)}\left[\liminf_{t\rightarrow +\infty} M_{t\wedge T}\right] \leq  \liminf_{t\rightarrow +\infty} \E_{(0,y)}\left[M_{t\wedge T}\right] \leq  \E_{(0,y)}\left[ M_0\right] =y$$
which proves that $T$ is integrable, i.e. $\gamma\geq1$. In particular, $\gamma> 1-\alpha/2$, which allows us to apply the maximal inequality of Pedersen \cite{Ped} recalled in (\ref{eq:Ped}). Take $\varepsilon>0$. By definition of $\gamma$,  $\E_{(0,y)}[T^{\gamma-\varepsilon}]<+\infty$, hence we have:
$$\E_{(0,y)}\left[\sup_{s\leq T} X_s^{\gamma-\varepsilon}\right] \leq \left(\frac{2(\gamma-\varepsilon)}{2(\gamma-\varepsilon) - (2-\alpha)} \right)^{\frac{2(\gamma-\varepsilon)}{2-\alpha}} \E_{(0,y)}\left[X_{T}^{\gamma-\varepsilon}\right]. $$
Using (\ref{eq:BDG}), we thus deduce from Fatou's lemma and the monotone convergence theorem that
$$c_{\alpha,\gamma} \E_{(0,y)}[T^\gamma] \leq c_{\alpha,\gamma} \liminf_{\varepsilon\rightarrow 0}\E_{(0,y)}[T^{\gamma-\varepsilon}] \leq  \left(\frac{2\gamma}{2\gamma - (2-\alpha)} \right)^{\frac{2\gamma}{2-\alpha}}  \left(1 +  \E_{(0,y)}\left[X_{T}^{\gamma}1_{\{X_T\geq1\}}\right] \right)<+\infty$$
since $\gamma<\theta$. As a consequence, the supremum in the definition of $\gamma$ is attained. 

\subsubsection{Second step} We shall now contradict the definition of $\gamma$ by showing that $T$ admits moments of order $\gamma+\varepsilon$ for some $\varepsilon>0$ small enough. To do so, notice first, still from the SDE (\ref{eq:SDE}), that since $\alpha>\beta$, the process $(\alpha Y_{t\wedge T}-  \beta X_{t\wedge T}, \; t\geq0)$ is a positive local martingale under $\Pb_{(0,y)}$.  
Let us set  
$$\displaystyle A_t =  \int_0^t (\beta^2 X_s + \alpha^2 Y_s) ds, \qquad t\geq0,$$
and observe that for $z>0$, the bounded process 
\begin{equation}\label{eq:expmart}
 \exp\left(z (\beta X_{t\wedge T}- \alpha Y_{t\wedge T}) - \frac{z^2}{2}A_{t\wedge T}\right), \qquad t\geq0,
 \end{equation}
is a positive martingale under $\Pb_{(0,y)}$. Applying Doob's optional stopping theorem and the dominated convergence theorem, we deduce that 
\begin{equation}\label{eq:Doob}
\E_{(0,y)}\left[e^{-z (\alpha - \beta) X_T - \frac{z^2}{2} A_T }\right] = e^{-z \alpha y}.
\end{equation}
The next lemma shows that the finiteness of the moments of $T$ is equivalent to those of $\sqrt{A_T}$.

\begin{lemma}\label{lem:TAT}
For $\lambda\in[0, \theta)$, 
$$\E_{(0,y)}[T^\lambda] <+\infty \qquad \Longleftrightarrow \qquad \E_{(0,y)}\left[A_T^{\frac{\lambda}{2}}\right]<+\infty.$$
\end{lemma}
\begin{proof}
On the one hand, applying Cauchy-Schwarz's inequality and Formula (\ref{eq:BDG}), 
\begin{align*}
\E_{(0,y)}\left[A_T^{\frac{\lambda}{2}}\right]&\leq \E_{(0,y)}\left[  T^{\frac{\lambda}{2}} \left(  2 \alpha^2 \sup_{s\leq T} Y_s    \right)^{\frac{\lambda}{2}}   \right]\\
&  \leq \E_{(0,y)}\left[  T^{\lambda}\right]^{\frac{1}{2}}   (\sqrt{2} \alpha)^\lambda \E_{(0,y)}\left[  \sup_{s\leq T}Y_s^\lambda\right]^{\frac{1}{2}} \leq  (\sqrt{2} \alpha)^\lambda \sqrt{C_{\alpha, \lambda} }\E_{(0,y)}\left[  T^{\lambda}\right]
\end{align*}
which proves  that if the expectation of $T^\lambda$ is finite, then so is the expectation of $A_T^{\frac{\lambda}{2}}$. On the other hand, going back to the SDE (\ref{eq:SDE}) defining $X$ and applying the Burkholder-Davis-Gundy inequality, we deduce that there exists a constant $K_\lambda$ such that 
\begin{align*}
\alpha^\lambda \E_{(0,y)}[T^\lambda] & \leq  \E_{(0,y)}\left[  \left(X_T  + \left|\int_0^T  \sqrt{X_s} dB_s\right|\right)^\lambda \right]   \\
&\leq K_\lambda \left(\E_{(0,y)}[X_T^\lambda]  +  \E_{(0,y)}\left[\left|\int_0^T  X_s ds\right|^\frac{\lambda}{2} \right]   \right)\\
& \leq K_\lambda \left(\E_{(0,y)}[X_T^\lambda]  +   \frac{1}{\alpha^\lambda} \E_{(0,y)}\left[A_T^\frac{\lambda}{2} \right]   \right)
\end{align*}
which concludes the proof of Lemma \ref{lem:TAT} since $X_T$ admits moments of order $\lambda<\theta$.
\end{proof}
Thanks to Lemma \ref{lem:TAT}, it is thus enough to prove that $\E_{(0,y)}\left[A_T^{\frac{\gamma+\varepsilon}{2}}\right]<+\infty$ for some $\varepsilon>0$ small enough such that $\gamma+\varepsilon<\theta$. We now set $n$ equal to the integer part of $\gamma$:
$n = \lfloor \gamma\rfloor$. Differentiating $n+1$ times Formula (\ref{eq:Doob}) with respect to $z$ and applying Leibniz rule, we obtain
\begin{equation}\label{eq:leib}
\sum_{k=0}^{n+1}  \binom{n+1}{k}   (\alpha-\beta)^{k}  \E_{(0,y)}\left[ X_T^{k}  A_T^{\frac{n+1-k}{2}}H_{n+1-k}(z \sqrt{A_T}) e^{-z (\alpha - \beta) X_T - \frac{z^2}{2} A_T }   \right]     =  (\alpha y)^{n+1} e^{-z \alpha y} 
\end{equation}
where the sequence $(H_k)_{k\geq 0}$ denotes the Hermite's polynomials, which are defined by:
$$H_n(x) = (-1)^n e^{\frac{x^2}{2}} \frac{\text{d}^n}{\text{d}x^n} e^{-\frac{x^2}{2}}.$$
 In particular, for any $k\in \N$, since $H_k$ is a polynomial, there exists a finite constant $c_k$ such that
$$ \sup_{x\geq 0} |H_k(x)| e^{-\frac{1}{4} x^2}  <c_k.$$
Take $\varepsilon$ small enough such that $n+1-\gamma-\varepsilon>0$ and $\gamma+\varepsilon<\theta$. We now integrate (\ref{eq:leib}) against $z^{n-\gamma-\varepsilon}$ on $(0,+\infty)$, applying repeatedly the Fubini-Tonelli theorem.
\begin{enumerate}
\item The right-hand side gives:
$$(\alpha y)^{n+1} \int_0^{+\infty}  z^{n-\gamma-\varepsilon}  e^{-z \alpha y}dz  =  \Gamma(n+1-\gamma-\varepsilon) (\alpha y)^{\gamma+\varepsilon}. $$
\item The term for $k=n+1$ gives:
\begin{multline*}
  (\alpha-\beta)^{n+1} \int_0^{+\infty}  z^{n-\gamma-\varepsilon}  \E_{(0,y)}\left[ X_T^{n+1} e^{-z (\alpha - \beta) X_T - \frac{z^2}{2} A_T }   \right]  dz\\ \leq  \Gamma(n+1-\gamma-\varepsilon) (\alpha-\beta)^{\gamma+\varepsilon} \E_{(0,y)}\left[X_T^{\gamma+\varepsilon} \right] <+\infty.
  \end{multline*}
\item The main sum, for $1\leq k \leq n$, gives:
\begin{align*}
&\left|\sum_{k=1}^{n}  \binom{n+1}{k}   (\alpha-\beta)^{k}  \int_0^{+\infty}  z^{n-\gamma-\varepsilon} \E_{(0,y)}\left[ X_T^{k}  A_T^{\frac{n+1-k}{2}}H_{n+1-k}(z \sqrt{A_T}) e^{-z (\alpha - \beta) X_T - \frac{z^2}{2} A_T }   \right] dz\right|\\
&\qquad \leq \sum_{k=1}^{n}  \binom{n+1}{k}   (\alpha-\beta)^{k}  c_{n+1-k}  \int_0^{+\infty}  z^{n-\gamma-\varepsilon} \E_{(0,y)}\left[ X_T^{k}  A_T^{\frac{n+1-k}{2}} e^{- \frac{z^2}{4} A_T}   \right]dz\\
&\qquad =  \Gamma\left(\frac{n+1-\gamma-\varepsilon}{2}\right) 2^{n-\gamma-\varepsilon }  \sum_{k=1}^{n}  \binom{n+1}{k} (\alpha-\beta)^{k}  c_{n+1-k}    \E_{(0,y)}\left[ X_T^{k}  A_T^{\frac{\gamma+\varepsilon-k}{2}}    \right].
\end{align*}
Furthermore, applying H\"older's inequality, the expectation in the sum on the right-hand side is smaller than
\begin{equation}\label{eq:Holder}
 \E_{(0,y)}\left[ X_T^{k}  A_T^{\frac{\gamma+\varepsilon-k}{2}}    \right]\leq \E_{(0,y)}\left[ X_T^{\frac{\gamma k}{k-\varepsilon}}\right]^{\frac{k-\varepsilon}{\gamma}}\E_{(0,y)}\left[  A_T^{\frac{\gamma}{2}}    \right]^{\frac{\gamma+\varepsilon-k}{\gamma}} <+\infty
 \end{equation}
which is finite provided that $\varepsilon$ is small enough so that $\gamma < \theta (1-\varepsilon)$.
\item As for the last term $k=0$, let us take $\delta>0$ and first write:
\begin{align*}
 &\left|\int_0^{+\infty} z^{n-\gamma-\varepsilon}  \E_{(0,y)}\left[A_T^{\frac{n+1}{2}}H_{n+1}(z \sqrt{A_T}) e^{-z (\alpha - \beta) X_T - \frac{z^2}{2} A_T } 1_{\{ \delta \sqrt{A_T}\leq  X_T\}}   \right] dz\right| \\
 &\qquad \leq  c_{n+1} \int_0^{+\infty} z^{n-\gamma-\varepsilon}  \E_{(0,y)}\left[A_T^{\frac{n+1}{2}} e^{-z (\alpha - \beta) X_T} 1_{\{ \delta \sqrt{A_T}\leq  X_T\}}   \right] dz \\
&\qquad = c_{n+1}  \frac{\Gamma(n+1-\gamma-\varepsilon) }{ (\alpha-\beta)^{n+1-\gamma-\varepsilon}}   \E_{(0,y)}\left[A_T^{\frac{n+1}{2}}  X_T^{-n-1+\gamma+\varepsilon} 1_{\{ \delta \sqrt{A_T}\leq  X_T\}}   \right] \\
&\qquad\leq  c_{n+1}  \frac{\Gamma(n+1-\gamma-\varepsilon) }{ (\alpha-\beta)^{n+1-\gamma-\varepsilon}}   \frac{1}{\delta^{n+1}}  \E_{(0,y)}\left[X_T^{\gamma+\varepsilon} \right] <+\infty.
\end{align*} 
\end{enumerate}
Plugging all the terms together in Equation (\ref{eq:leib}), we have thus proven that for $\varepsilon>0$ small enough
\begin{equation}\label{eq:intfinite}
\left|\int_0^{+\infty} z^{n-\gamma-\varepsilon}  \E_{(0,y)}\left[A_T^{\frac{n+1}{2}}H_{n+1}(z \sqrt{A_T}) e^{-z (\alpha - \beta) X_T - \frac{z^2}{2} A_T } 1_{\{ \delta \sqrt{A_T}>  X_T\}}   \right] dz\right| <+\infty.
\end{equation}
We now study further this last expression, and assume, without loss of generality, that 
\begin{equation}\label{eq:delta}
\int_0^{+\infty} z^{n-\gamma-\varepsilon} H_{n+1}(z) e^{-\frac{z^2}{2}} dz >0.
\end{equation}
Applying the Fubini-Tonelli theorem and a change of variables, we have:
\begin{align*}
&\int_0^{+\infty}z^{n-\gamma-\varepsilon}\E_{(0,y)}\left[A_T^{\frac{n+1}{2}}H_{n+1}(z \sqrt{A_T}) 1_{\{H_{n+1}(z \sqrt{A_T})>0\}}e^{-z (\alpha - \beta) X_T - \frac{z^2}{2} A_T } 1_{\{ \delta \sqrt{A_T}>  X_T\}}  \right]  dz\\
&\qquad\qquad =\E_{(0,y)}\left[A_T^{\frac{\gamma+\varepsilon}{2}}1_{\{ \delta \sqrt{A_T}>  X_T\}} \int_0^{+\infty}z^{n-\gamma-\varepsilon}H_{n+1}(z ) 1_{\{H_{n+1}(z)>0\}}e^{-z (\alpha - \beta) \frac{X_T}{\sqrt{A_T}} - \frac{z^2}{2} }  dz  \right] \\
&\qquad\qquad \geq \E_{(0,y)}\left[A_T^{\frac{\gamma+\varepsilon}{2}}1_{\{ \delta \sqrt{A_T}>  X_T\}}  \right]  \int_0^{+\infty}z^{n-\gamma-\varepsilon}H_{n+1}(z ) 1_{\{H_{n+1}(z)>0\}}e^{-z (\alpha - \beta) \delta - \frac{z^2}{2} }  dz. 
\end{align*}
Similarly,
\begin{align*}
&\int_0^{+\infty}z^{n-\gamma-\varepsilon}\E_{(0,y)}\left[A_T^{\frac{n+1}{2}}|H_{n+1}(z \sqrt{A_T})| 1_{\{H_{n+1}(z \sqrt{A_T})<0\}}e^{-z (\alpha - \beta) X_T - \frac{z^2}{2} A_T } 1_{\{ \delta \sqrt{A_T}>  X_T\}}  \right]  dz\\
&\qquad\qquad \leq \E_{(0,y)}\left[A_T^{\frac{\gamma+\varepsilon}{2}}1_{\{ \delta \sqrt{A_T}>  X_T\}}  \right]  \int_0^{+\infty}z^{n-\gamma-\varepsilon}|H_{n+1}(z )| 1_{\{H_{n+1}(z)<0\}}e^{- \frac{z^2}{2} }  dz. 
\end{align*}
As a consequence, we deduce from (\ref{eq:intfinite}) that
$$ \E_{(0,y)}\left[A_T^{\frac{\gamma+\varepsilon}{2}}1_{\{ \delta \sqrt{A_T}>  X_T\}}  \right] \int_0^{+\infty} z^{n-\gamma-\varepsilon}  |H_{n+1}(z )| e^{-\frac{z^2}{2}} \left( 1_{\{H_{n+1}(z)>0\}} e^{-z (\alpha - \beta) \delta} -1_{\{H_{n+1}(z)<0\}} \right) dz<+\infty.$$
But, from (\ref{eq:delta}) the integral in $z$ is strictly positive for $\delta$ small enough. As a consequence, we have obtained that 
$$ \E_{(0,y)}\left[A_T^{\frac{\gamma+\varepsilon}{2}} 1_{\{ \delta \sqrt{A_T}>  X_T\}}  \right]<+\infty$$
and the result follows from the observation that
$$\E_{(0,y)}\left[A_T^{\frac{\gamma+\varepsilon}{2}}\right] \leq  \E_{(0,y)}\left[A_T^{\frac{\gamma+\varepsilon}{2}} 1_{\{ \delta \sqrt{A_T}>  X_T\}}  \right]  +  \frac{1}{\delta^{\gamma+\varepsilon}} \E_{(0,y)}\left[X_T^{\gamma+\varepsilon} 1_{\{ \delta \sqrt{A_T} \leq  X_T\}}  \right] <+\infty.$$
From Lemma \ref{lem:TAT}, this contradicts the definition of $\gamma$, hence we conclude that $\gamma=\theta$. \qed 

\subsubsection{The expectation of $T$}
We show in this section how to compute the expectation of $T$ (which is finite since $\alpha>\beta$, i.e. $\theta>1$) by using the martingale $(Y_t-X_t - (\alpha-\beta)t, \, t\geq0)$.
Let $n>0$ and recall the definition of the stopping time $\tau_n^Y = \inf\{t\geq0,\; Y_t =n\}$. From Doob's optional theorem:
$$\E_{(x,y)}\left[ Y_{t\wedge T\wedge \tau_n^Y} -X_{t\wedge T\wedge \tau_n^Y}\right] - (\alpha-\beta)\E_{(x,y)}\left[t\wedge T\wedge \tau_n^Y \right] = y-x.$$
Using the monotone and dominated convergence theorems, we deduce:
\begin{align*}
\E_{(x,y)}\left[ T \right] &= \frac{y-x}{\alpha-\beta} - \frac{1}{\alpha-\beta} \lim_{n\rightarrow +\infty}\E_{(x,y)}\left[ Y_{T\wedge \tau_n^Y} -X_{T\wedge \tau_n^Y}\right] \\
&= \frac{y-x}{\alpha-\beta} - \frac{1}{\alpha-\beta} \lim_{n\rightarrow +\infty}\E_{(x,y)}\left[ (n -X_{\tau_n^Y}) 1_{\{\tau_n^Y \leq T\}}\right] 
\end{align*}
and it remains to show that the limit equals 0. Let us take $0<\varepsilon < \theta-1$. Applying the Markov inequality together with the maximal inequality (\ref{eq:BDG}), we obtain: \begin{align*}
0\leq \E_{(x,y)}\left[ (n -X_{\tau_n^Y}) 1_{\{\tau_n^Y \leq T\}}\right] & \leq  n  \Pb_{(x,y)}\left(  \tau_n^Y\leq T\right)\\
 &= n \Pb_{(x,y)}\left(\sup_{s\leq T} Y_s \geq n  \right)\\
&\leq n^{-\varepsilon} \E_{(x,y)}\left[\sup_{s\leq T} Y_s^{1+\varepsilon}  \right] \leq  C_{\alpha, 1+\varepsilon}  n^{-\varepsilon} \E_{(x,y)}\left[(T+x)^{1+\varepsilon}\right] \xrightarrow[n\rightarrow +\infty]{}0
\end{align*}
since the last expectation is finite by the first part of the proof.\qed

\section{Proof of Corollary \ref{cor:1}}\label{sec:5}

Before proving Corollary  \ref{cor:1}, we first study the asymptotics of $X_T$.
Recall to this end that from Proposition \ref{prop:X_T} and from the analyticity of $F_{\alpha,\beta}$, there exists $m\in \N$ such that
\begin{equation}\label{eq:XTm}
\E_{(0,y)}\left[X_T^{\theta-s}\right] = \frac{(2y)^{\theta-s}}{(-s)^m G_{\alpha,\beta}(\theta-s)}
\end{equation}
where $G_{\alpha,\beta}$ is such that $G_{\alpha,\beta}(\theta)\neq0$.

\subsection{Asymptotics of the tail distribution of $X_T$}
 
\begin{lemma}\label{lem:AsympXT}
Let $\delta>0$. There exist two positive constants $\kappa_1$ and $\kappa_2$ such that
$$\frac{\kappa_1}{t^{\theta+\delta}} \leq \Pb_{(0,y)}\left(X_T> t\right) \leq \kappa_2 \frac{(\ln(t))^m}{t^{\theta}}, \qquad \text{as }t\rightarrow +\infty.$$
\end{lemma}  

\begin{proof}
The upper bound is a direct consequence of the Markov inequality, for $t\geq2$, 
$$\Pb_{(0,y)}\left(X_T>t\right) \leq  \frac{\E_{(0,y)}\left[X_T^{\theta-\frac{1}{\ln(t)}}\right]}{t^{\theta-\frac{1}{\ln(t)}}} = \frac{ e}{t^\theta}\E_{(0,y)}\left[X_T^{\theta-\frac{1}{\ln(t)}}\right]$$
together with Formula (\ref{eq:XTm})
\begin{equation}\label{eq:XTln}
\E_{(0,y)}\left[X_T^{\theta-\frac{1}{\ln(t)}}\right] = \frac{ (2y)^{\theta-\frac{1}{\ln(t)}}}{(-1)^m G_{\alpha,\beta}(\theta-\frac{1}{\ln(t)})}(\ln(t))^m.
\end{equation}
To get the lower bound, we shall write the Mellin transform of $X_T$ as a Laplace transform:
\begin{align*}
\frac{1}{\theta-s} \E_{(0,y)}[X_T^{\theta-s}] &=  \int_0^{+\infty}  z^{\theta-s+1} \Pb_{(0,y)}\left(X_T> z\right) dz\\
&=  \int_{-\infty}^{+\infty}   e^{-s z }e^{\theta z} \Pb_{(0,y)}\left(X_T >e^z\right) dz \\
& =  \int_{-\infty}^{0}   e^{-s z }e^{\theta z} \Pb_{(0,y)}\left(X_T >e^z\right) dz  +  \int_{0}^{+\infty}   e^{-s z }e^{\theta z} \Pb_{(0,y)}\left(X_T >e^z\right) dz. 
\end{align*}
We now let $s\downarrow 0$. Applying the monotone convergence theorem, the first integral on the right-hand side converges towards $ \int_{-\infty}^{0} e^{\theta z} \Pb_{(0,y)}\left(X_T >e^z\right) dz  \leq \dfrac{1}{\theta}$. As a consequence, applying  Karamata's Tauberian theorem \cite[Theorem 1.7.1]{BGT}, we deduce from  (\ref{eq:XTm}) that there exists a constant $c>0$ such that 
$$    \int_0^z e^{\theta u} \Pb_{(0,y)}\left(X_T > e^u\right) du  \equi_{z\rightarrow +\infty} c z^m $$
 i.e., going back to the original variable
$$  \int_0^z  t^{\theta-1}  \Pb_{(0,y)}\left(X_T > t\right)dt  \equi_{z\rightarrow +\infty}  c \left(\ln(z)\right)^m.$$
We now fix $\delta>0$ and take $\varepsilon>0$ small enough such that $c(1-\varepsilon)(1+\delta)^m - c(1+\varepsilon)>0$. Then, for $z>0$ large enough, we have 
\begin{align*}
\frac{1}{\theta}   z^{\theta(1+\delta)}  \Pb_{(0,y)}\left(X_T> z\right)& \geq\int_z^{z^{1+\delta}}  t^{\theta-1}  \Pb_{(0,y)}\left(X_T >t\right)dt\\
& \geq  c(1-\varepsilon) \left(\ln(z^{1+\delta})\right)^m - c(1+\varepsilon) \left(\ln(z)\right)^m\\
&\geq \big(c(1-\varepsilon)(1+\delta)^m - c(1+\varepsilon) \big) \left(\ln(z)\right)^m\xrightarrow[z\rightarrow+\infty]{}+\infty
\end{align*}
which yields the lower bound of Lemma \ref{lem:AsympXT}.
\end{proof}

\subsection{Proof of Corollary \ref{cor:1}}

Observe first that applying the Markov inequality and the maximal inequality (\ref{eq:Ped}), there exists a constant $K>0$ such that for all $t\geq 2$,
$$
\Pb_{(0,y)}\left(T>t\right) \leq \frac{\E_{(0,y)}\left[T^{\theta- \frac{1}{\ln(t)}}\right]}{t^{\theta-\frac{1}{\ln(t)}}} \leq K \frac{\E_{(0,y)}\left[X_T^{\theta-\frac{1}{\ln(t)}}\right]}{t^{\theta-\frac{1}{\ln(t)}}} = \frac{K e}{t^\theta}\E_{(0,y)}\left[X_T^{\theta-\frac{1}{\ln(t)}}\right]
$$
and the upper bound follows as above from (\ref{eq:XTln}). 
Then, taking $\varepsilon>0$ small enough, we have 
$$
 \mathbb{P}_{(0,y)}(X_T > t) \leq \mathbb{P}_{(0,y)}\left( \sup_{[0,T]} X_s > t \right) \leq \mathbb{P}_{(0,y)}(T > t^{1 - \varepsilon}) + \mathbb{P}_{(0,y)}\left( \sup_{[0, t^{1 - \varepsilon}]} X_s > t \right).
 $$
By scale invariance, the last term on the right-hand side equals $\mathbb{P}_{(0,y)}\left( \sup_{[0,1]} X_s > t^\varepsilon \right) $ which decreases exponentially since $\sup_{[0,1]} X_s $ admits exponential moments, see \cite{Eis}.
As a consequence, we deduce from Lemma \ref{lem:AsympXT} that there exists a constant $c_1>0$ such that
$$\frac{c_1}{t^{\frac{\theta+\delta}{1-\varepsilon}}}  \leq \Pb_{(0,y)}\left( T>t\right)\qquad \text{ as }t\rightarrow +\infty.$$
This is the lower bound of Corollary \ref{cor:1}, after renaming the constants.\qed

\section{The limit of $\theta(\alpha,\beta)$ as $\alpha\downarrow0$.}\label{sec:limalpha}

When $\alpha=0$, the process $X$ is absorbed at 0. As a consequence, as observed in Figure 2, different situations occur,  according as whether $Y$ may or may not reach the level 0.
Typically, when $\beta\geq 2$, the process $Y$ cannot reach 0, so that $\Pb_{(x,y)}(T=+\infty)>0$, and $\lim_{\alpha\downarrow0}\theta(\alpha,\beta)=0$. Conversely, when $\beta<2$, the process $Y$ reaches 0 a.s., and thus the random variable $X_T$ remains well-defined, but its distribution admits an atom at 0 (on the set $\{T\geq \tau^X_0\}$). We compute below the limit of $\theta(\alpha,\beta)$ as $\alpha\downarrow0$ in these different regimes. Note that this limit necessarily exists since $\theta(\alpha,\beta)$ is non-decreasing in $\alpha$.

\begin{proposition}
The following limits hold as $\alpha\downarrow0$ :
$$
\begin{array}{ll}
\theta(\alpha,\beta) \equi\alpha / (2^{\frac{\beta}{2}}-2) &\qquad \text{if }\beta>2,\\
\theta(\alpha,\beta) \equi  \sqrt{\frac{\alpha}{2\ln(2)}}&\qquad \text{if }\beta=2,\\
\theta(\alpha,\beta) = 1-\frac{\beta}{2} +\alpha \left( \frac{1}{2-2^{\beta/2}}-\frac{1}{2} \right) +  \emph{o}(\alpha)&\qquad \text{if }0\leq \beta<2.\\
\end{array}$$
\end{proposition}

\begin{proof}

By definition, $\theta(\alpha,\beta)$ is a solution of the equation:
    \begin{align}
  \notag  0 = F_{\alpha, \beta}(\theta(\alpha, \beta)) &= 1+ \sum_{n=1}^{+\infty}  \frac{\left(\frac{\alpha+\beta}{2} - 1 + \theta(\alpha,\beta)\right)_n  (-\theta(\alpha,\beta)) (1-\theta(\alpha,\beta))_{n-1} }{\frac{\alpha}{2}\left(\frac{\alpha}{2}+1\right)_{n-1} n!} \left(\frac{1}{2}\right)^n\\
\label{eq:Ftheta}&=1-2 \frac{\theta(\alpha,\beta)}{\alpha}  \sum_{n=1}^{+\infty}  \frac{\left(\frac{\alpha+\beta}{2} - 1 + \theta(\alpha,\beta)\right)_n (1-\theta(\alpha,\beta))_{n-1} }{\left(\frac{\alpha}{2}+1\right)_{n-1} n!} \left(\frac{1}{2}\right)^n.
\end{align}
Also, recall that $\theta(\alpha,\beta) \in (0,1)$ for $\alpha<\beta$. We deal with each case $\beta>2$, $\beta=2$ and $\beta<2$ separately.

\begin{enumerate}
\item When $\beta>2$, we have $\lim\limits_{\alpha\downarrow0}\theta(\alpha,\beta)=0$. Passing to the limit in (\ref{eq:Ftheta}), we deduce that the sum converges towards
$$\sum_{n=1}^{+\infty}  \frac{\left(\frac{\beta}{2} - 1\right)_n }{ n! } \left(\frac{1}{2}\right)^n = 2^{\frac{\beta}{2}-1} -1$$
and thus 
$$\theta(\alpha,\beta) \equi_{\alpha\rightarrow 0}  \frac{\alpha}{2^{\frac{\beta}{2}}-2}.  $$
Note that when $\beta =4$, we obtain $\theta(\alpha,4) \equi_{\alpha\downarrow 0} \dfrac{\alpha}{2}$ which is consistent with the explicit value $\theta(\alpha, 4-\alpha) = \alpha/2$ obtained in the case $\alpha+\beta=4$.\\
\item When $\beta=2$, we still have $\lim\limits_{\alpha\downarrow0}\theta(\alpha,\beta)=0$. Equation (\ref{eq:Ftheta}) then reads
$$ 0 = 1-2 \frac{\theta(\alpha,2)\left(\frac{\alpha}{2}  + \theta(\alpha,2)\right)  }{\alpha}  \sum_{n=1}^{+\infty}  \frac{\left(\frac{\alpha}{2}  + \theta(\alpha,2)+1\right)_{n-1} (1-\theta(\alpha,2))_{n-1} }{\left(\frac{\alpha}{2}+1\right)_{n-1} n!} \left(\frac{1}{2}\right)^n.$$
As before, passing to the limit as $\alpha\downarrow0$, the sum converges towards
$$\sum_{n=1}^{+\infty}  \frac{ 1 }{n} \left(\frac{1}{2}\right)^n=\ln(2)$$
and we obtain
$$0 = 1 - 2\ln(2)  \lim_{\alpha\downarrow0}\frac{\theta^2(\alpha,2)}{\alpha} $$
i.e.
$$\theta(\alpha,2) \equi_{\alpha\downarrow 0}  \sqrt{\frac{\alpha}{2\ln(2)}}.$$
\item When $0\leq \beta<2$, then $Y$ will hit 0 a.s. and we have $T\leq \tau_0^{Y}$. As explained in Subsection \ref{sec:4}, this implies that $\theta(\alpha,\beta)> 1-\frac{\beta}{2}$. Let us set $\theta(\alpha,\beta)=1-\frac{\beta}{2} + \rho(\alpha,\beta)$ so that Equation (\ref{eq:Ftheta}) becomes
\begin{equation}\label{eq:=0}
0=1-2\left( \frac{\frac{\alpha}{2} + \rho(\alpha,\beta)}{\alpha}\right)\theta(\alpha,\beta)  \sum_{n=1}^{+\infty}  \frac{\left(\frac{\alpha}{2} + \rho(\alpha,\beta)+1\right)_{n-1} (1-\theta(\alpha,\beta))_{n-1} }{\left(\frac{\alpha}{2}+1\right)_{n-1} n!} \left(\frac{1}{2}\right)^n.
\end{equation}
Passing to the limit as $\alpha\downarrow0$, the sum converges towards 
$$\sum_{n=1}^{+\infty}  \frac{\left( \rho(0^+,\beta)+1\right)_{n-1} (1-\theta(0^+,\beta))_{n-1} }{(n-1)! n!} \left(\frac{1}{2}\right)^n>0.$$
As a consequence, going back to (\ref{eq:=0}), we necessarily have $\rho(0^+,\beta)=0$ and the sum equals 
$$\sum_{n=1}^{+\infty}  \frac{\left(\frac{\beta}{2}\right)_{n-1} }{n!} \left(\frac{1}{2}\right)^n =  \frac{1}{\frac{\beta}{2}-1} \sum_{n=1}^{+\infty}  \frac{\left(\frac{\beta}{2}-1\right)_{n}}{n!} \left(\frac{1}{2}\right)^n  =   \frac{2}{\beta-2}  \left(  2^{\frac{\beta}{2}-1}-1\right)$$
so finally
$$\rho(\alpha,\beta) \equi_{\alpha\downarrow 0} \alpha \left( \frac{1}{2-2^{\beta/2}}-\frac{1}{2} \right).$$
\end{enumerate}

\end{proof}

\section{Appendix}

We write down a formula which is used several times in the paper. This identity may also be recovered from \cite[p.329, Formula 3.259 (3)]{GrRy}.
\begin{lemma}\label{lem:app}
Assume that $\lambda>0$, $\mu>0$  and $ 0<\gamma < \lambda+\mu$. Then:
\begin{multline*}
\int_0^{+\infty} \xi^{\gamma-1} (1-2\i \xi)^{-\lambda} (1+2\i\xi )^{-\mu} d\xi\\
=\frac{\Gamma(\gamma) \Gamma(\mu+\lambda-\gamma)}{\Gamma(\lambda)\Gamma(\mu)} \left((2\i)^{-\gamma} B\left(\lambda, 1-\gamma  \right)2^{-\lambda} \pFq{2}{1}{1-\mu \quad\lambda}{\lambda -\gamma + 1 }{\frac{1}{2}} \right.\\
\left.+ (-2\i)^{-\gamma}B\left( \mu, 1-\gamma \right)  2^{-\mu} \pFq{2}{1}{1-\lambda \quad\mu}{\mu -\gamma + 1 }{\frac{1}{2}}\right).
\end{multline*}
\end{lemma}

\begin{proof}
We start by writing the left-hand side under the form 
\begin{multline*}
\int_0^{+\infty} \xi^{\gamma-1} (1-2\i \xi)^{-\lambda} (1+2\i\xi )^{-\mu} d\xi\\=  \frac{1}{\Gamma(\lambda)\Gamma(\mu)} \int_0^{+\infty} \xi^{\gamma-1}  \left(\int_0^{+\infty}  a^{\lambda-1} e^{- a (1-2\i\xi)} da\right) \left(\int_0^{+\infty} b^{\mu-1} e^{-b(1+2\i \xi)} db\right)d\xi.
\end{multline*}
Assume first that $a<b$:
\begin{align*}
&\frac{1}{\Gamma(\lambda)\Gamma(\mu)} \int_0^{+\infty} \xi^{\gamma-1}  \int_0^{+\infty}  b^{\mu-1} e^{- b (1+2\i\xi)} \int_0^b  a^{\lambda-1} e^{-a(1-2\i \xi)} db da d\xi\\
&\qquad=\frac{\Gamma(\gamma)}{\Gamma(\lambda)\Gamma(\mu)}  \int_0^{+\infty}  (2\i (b-a))^{-\gamma}  b^{\mu-1} e^{- b } \int_0^b  a^{\lambda-1} e^{-a} da db \\
&\qquad=\frac{\Gamma(\gamma) (2\i)^{-\gamma}}{\Gamma(\lambda)\Gamma(\mu)}  \int_0^{+\infty}   (1-a)^{-\gamma}  b^{\mu+\lambda-\gamma- 1} e^{- b } \int_0^1  a^{\lambda-1} e^{-ab} da db \\
&\qquad=\frac{\Gamma(\gamma) (2\i)^{-\gamma}}{\Gamma(\lambda)\Gamma(\mu)} \Gamma(\mu+\lambda-\gamma)    \int_0^1  (1-a)^{-\gamma} (1+a)^{\gamma-\lambda-\mu}    a^{\lambda-1}  da \\
&\qquad=\frac{\Gamma(\gamma) (2\i)^{-\gamma}}{\Gamma(\lambda)\Gamma(\mu)} \Gamma(\mu+\lambda-\gamma)  B\left(\lambda, 1-\gamma  \right)\pFq{2}{1}{\mu+\lambda-\gamma\quad\lambda}{\lambda -\gamma + 1 }{-1}
\end{align*}
Recalling  the Pfaff transformation 
$$\pFq{2}{1}{\mu+\lambda-\gamma\quad\lambda}{\lambda -\gamma + 1 }{-1}= 2^{-\lambda} \pFq{2}{1}{1-\mu \quad\lambda}{\lambda -\gamma + 1 }{\frac{1}{2}}$$
yields the first term on the right-hand side. Similarly, when $b<a$, we obtain:
\begin{align*}
&\frac{1}{\Gamma(\lambda)\Gamma(\mu)} \int_0^{+\infty} \xi^{\gamma-1}  \int_0^{+\infty}  b^{\mu-1} e^{- b (1+2\i\xi)} \int_b^{+\infty}  a^{\lambda-1} e^{-a(1-2\i \xi)} db da d\xi\\
&\qquad=\frac{\Gamma(\gamma) (-2\i)^{-\gamma}}{\Gamma(\lambda)\Gamma(\mu)} \Gamma(\mu+\lambda-\gamma)    \int_1^{+\infty}  (a-1)^{-\gamma} (1+a)^{\gamma-\lambda-\mu}    a^{\lambda-1}  da \\
&\qquad=\frac{\Gamma(\gamma) (-2\i)^{-\gamma}}{\Gamma(\lambda)\Gamma(\mu)} \Gamma(\mu+\lambda-\gamma)    \int_0^{1}  b^{\mu-1} (1-b)^{-\gamma} (1+b)^{\gamma-\lambda-\mu}
db \\
&\qquad=\frac{\Gamma(\gamma) (-2\i)^{-\gamma}}{\Gamma(\lambda)\Gamma(\mu)} \Gamma(\mu+\lambda-\gamma) B\left( \mu, 1-\gamma \right)  \pFq{2}{1}{\mu+\lambda-\gamma\quad\mu}{\mu-\gamma+1}{-1}\\
&\qquad=\frac{\Gamma(\gamma) (-2\i)^{-\gamma}}{\Gamma(\lambda)\Gamma(\mu)} \Gamma(\mu+\lambda-\gamma) B\left( \mu, 1-\gamma \right)  2^{-\mu} \pFq{2}{1}{1-\lambda \quad\mu}{\mu -\gamma + 1 }{\frac{1}{2}}
\end{align*}
and Lemma \ref{lem:app} follows by summing both terms.
\end{proof}

\textbf{Acknowledgments.} We wish to thank the two referees for their careful reading and for their many suggestions, which helped improve the content of the paper and simplify some of the proofs, in particular the last part of Corollary \ref{cor:1}.

\end{document}